\documentclass[review]{elsarticle}

\usepackage{lineno}
\usepackage[hidelinks]{hyperref}
\usepackage{amsmath}
\usepackage{amsfonts}
\usepackage{amsthm}
\newtheorem{theorem}{Theorem}
\newtheorem{lemma}{Lemma}

\modulolinenumbers[5]

\journal{Theoretical Computer Science}

\begin{document}

\begin{frontmatter}

\title{The Iris function and the matrix permanent}

\author{Ali \"{O}nder Bozdo\u{g}an}
\address{Ankara \"{U}niversitesi Elektrik-Elektronik M\"{u}hendisli\u{g}i B\"{o}l\"{u}m\"{u}, 06830, Ankara, T\"{u}rkiye}




\begin{abstract}
This paper defines the Iris function and provides two formulations of the matrix permanent. The first formulation, valid for arbitrary complex matrices, expresses the permanent of a complex matrix as a contour integral of a second order Iris function over the unit circle around zero. The second formulation is defined for the restricted set of matrices with complex or ``Gaussian" integer elements. Using the second formulation, the paper shows that the computation of the permanent of an arbitrary $n\times n$ 0-1 matrix is bounded by  $o\left( {{n^{27}}{{\left( {\log \left( {{n^3}} \right)} \right)}^6}{{\left( {{{\log }_2}\left( n \right)} \right)}^2}} \right)$  binary operations.  
\end{abstract}

\begin{keyword}
Iris function, permanent of complex matrices, permanent of complex integer matrices, $\#P$-complete.
\end{keyword}

\end{frontmatter}

\linenumbers

\section{Introduction}

 Let $A$ be an $n\times n$ matrix with elements ${A_{.,.}} \in \mathbb{C},$ and,  for $i \le k \in \mathbb{Z}$, let  $Z_{i,k}$ denote the set of consecutive integers from $i$ to $k$, i.e. $Z_{i,k}=\left\{ {i, \ldots ,k} \right\}$. The permanent of $A$ is defined as
\begin{equation}\label{eq:i1}
{\rm{per}}{\mkern 1mu} \left( A \right) = \sum\limits_{\sigma  \in {S_n}} {\prod\limits_{i = 1}^n {{A_{i,\sigma_i}}} } ,
\end{equation}
where $S_n$ is the set of column vectors of dimension $n$, such that, 
\begin{equation}\label{eq:sn}
{S_n} = \left\{ \begin{array}{l}
\sigma :\sigma  = {\left[ {\begin{array}{*{20}{c}}
{{\sigma _1}}& \ldots &{{\sigma _n}}
\end{array}} \right]^T,}\,\,\sigma_i\in {Z_{1,n}}\,\, \forall \,i \in {Z_{1,n}},\\
\,\,\forall \,i \in {Z_{1,n}}{\rm{ }}\,{\rm{and }}\,\,\forall j \in {Z_{1,n}}\,\,\,i \ne j \to {\sigma _i} \ne {\sigma _j}\,{\rm{ }}
\end{array} \right\}.
\end{equation}
 The elements of $S_n$ are vectors of the permutations of the first positive $n$ integers.
 
The generating function of the matrix permanent is given and a contour integral formulation of the permanent is derived in \cite{roySaddle}. Let $z_k\in \mathbb{C}$ for $k \in Z_{1,n}$, the generating function $\Psi^A :\mathbb{C}^n \to \mathbb{C}$ of the permanent of matrix $A$ is given by 
   \begin{equation}\label{eq:permanentGeneratingFun}
\Psi^A \left( {{z_1}, \ldots ,{z_n}} \right) = \prod\limits_{i = 1}^n {\sum\limits_{k = 1}^n {{A_{i,k}}{z_k}} }. 
   \end{equation}
Observe that $\Psi^A\left(.\right)$ is entire. From \eqref{eq:permanentGeneratingFun}, the permanent of matrix $A$ is calculated as 
   \begin{equation}\label{eq:permanentGeneratingFunDiff}
{\rm{per}}\left( A \right) = {\left. {\frac{{{d^n}}}{{d{z_1} \ldots d{z_n}}}\Psi^A \left( {{z_1}, \ldots ,{z_n}} \right)} \right|_{{z_1} =  \ldots  = {z_n} = 0}}.
   \end{equation}
Since $\Psi^A$ is entire,  using the Cauchy's integral formula for several variables for $\Psi^A$ around circular contours with positive radii gives (see \cite{bersComplex}, Theorem 2)  
   \begin{equation}\label{eq:permanentGeneratingFunCauchy}
\Psi^A \left( {{\zeta _1}, \ldots ,{\zeta _n}} \right) = \frac{1}{{{{\left( {2\pi j} \right)}^n}}}\oint\limits_{\left| {{z_1} - {\zeta _1}} \right| = {r_1}}  \cdots  \oint\limits_{\left| {{z_n} - {\zeta _n}} \right| = {r_n}} {\frac{{\Psi \left( {{z_1}, \ldots ,{z_n}} \right)}}{{\left( {{z_1} - {\zeta _1}} \right) \cdots \left( {{z_n} - {\zeta _n}} \right)}}d} {z_1} \cdots d{z_n},
   \end{equation}    
where $j^2=-1$. From \eqref{eq:permanentGeneratingFunDiff}, differentiating inside the integral and setting $\zeta_.=0$ gives
   \begin{equation}\label{eq:permanentGeneratingFunCauchyDiff}
{\rm{per}}\left( A \right) = \frac{1}{{{{\left( {2\pi j} \right)}^n}}}\oint\limits_{\left| {{z_1}} \right| = {r_1}}  \cdots  \oint\limits_{\left| {{z_n}} \right| = {r_n}} {\frac{{\Psi^A \left( {{z_1}, \ldots ,{z_n}} \right)}}{{{z_1}^2 \cdots {z_n}^2}}d} {z_1} \cdots d{z_n}.
   \end{equation} 
Set $z_. = \exp \left( {j{\theta _.}}\right)$ for  $0\le \theta_.<2\pi $, then \eqref{eq:permanentGeneratingFunCauchyDiff} can be rewritten as 
   \begin{equation}\label{eq:permanentGeneratingFunCauchyDiffExp}
\begin{array}{*{20}{l}}
{{\rm{per}}\left( A \right) = \frac{1}{{{{\left( {2\pi } \right)}^n}}}\int\limits_0^{2\pi }  \cdots  \int\limits_0^{2\pi } {\frac{{\Psi \left( {\exp \left( {j{\theta _1}} \right), \ldots ,\exp \left( {j{\theta _n}} \right)} \right)}}{{\exp \left( {j{\theta _1}} \right) \cdots \exp \left( {j{\theta _n}} \right)}}d{\theta _1} \cdots d{\theta _n}} }\\
{ = \frac{1}{{{{\left( {2\pi } \right)}^n}}}\int\limits_0^{2\pi }  \cdots  \int\limits_0^{2\pi } {\frac{{\prod\limits_{i = 1}^n {\sum\limits_{k = 1}^n {{A_{i,k}}\exp \left( {j{\theta _k}} \right)} } }}{{\prod\limits_{i = 1}^n {\exp \left( {j{\theta _i}} \right)} }}d{\theta _1} \cdots d{\theta _n}} }.
\end{array}
   \end{equation}    
Observe that \eqref{eq:permanentGeneratingFunCauchyDiffExp} is  an alternative, but equivalent, way of expressing \eqref{eq:i1}. This is because, the only terms that are free of a complex exponential in the integrand of \eqref{eq:permanentGeneratingFunCauchyDiffExp} are those terms in \eqref{eq:i1}. Any other, or alias, term in the integrand must be accompanied by a complex exponential, which, then, must be eliminated by the integral itself. Naively, \eqref{eq:permanentGeneratingFunCauchyDiffExp} can be evaluated as 
   \begin{equation}\label{eq:permanentGeneratingFunCauchyDiffExpSumForm}
{\rm{per}}\left( A \right) = \frac{1}{{{n^n}}}\sum\limits_{{i_1} = 0}^{n - 1} { \cdots \sum\limits_{{i_n} = 0}^{n - 1} {\left( {\frac{{\prod\limits_{k = 1}^n {\sum\limits_{l = 1}^n {{A_{k,l}}\exp \left( {j\frac{{2\pi {i_l}}}{n}} \right)} } }}{{\prod\limits_{l = 1}^n {\exp \left( {j\frac{{2\pi {i_l}}}{n}} \right)} }}} \right)} } ,
\end{equation}  
where the nested integral is evaluated over the discrete grid in which each dimension is sampled with $n$ samples, thus require an exponential computational complexity. The best known exact method to calculate the permanent \cite{ryserBook} (Chapter 2, Theorem 4.1) takes ${2^{n + O\left( {\log \left( n \right)} \right)}}$ steps, which, still is of an exponential complexity, but with a lower factor for each dimension. This, in principle, indicates that the exponential complexity of calculating the permanent should be related to, if not a direct result of, what is generally known as the curse of dimensionality of integration. Thus, a polynomial time algorithm that calculates the permanent of an arbitrary matrix, if written in an integral form, should be composed of a fixed, or slowly increasing number of nested integrals with respect to the dimension of the matrix.

Define the Iris function of order $t\ge 1$ for an $n\times n $ complex matrix $A$, $\iota _t^A: \mathbb{C}^t \times \mathbb{Z^+}^{t\times n} \to \mathbb{C}$
   \begin{equation}\label{eq:Iris_t}
   {\iota _t^A}\left( {{z_1}, \ldots ,{z_t}},\alpha \right) = \prod\limits_{i = 1}^n {\sum\limits_{k = 1}^n {\left( {{A_{i,k}}\prod\limits_{l = 1}^t {z_l^{\alpha _k^l}} } \right)} } ,
\end{equation}  
where the row vector ${\alpha ^l}$ denotes the $l^{th}$ row of the matrix $\alpha \in \mathbb{Z^+}^{t\times n}$, i.e. ${\alpha ^l} = \left[ {\begin{array}{*{20}{c}}
{\alpha _1^l}& \ldots &{\alpha _n^l}
\end{array}} \right]$ for $l \in Z_{1,t}$. The matrix $\alpha$ should only be selected such that the permanent of matrix $A$ is equal to the mixed derivative
   \begin{equation}\label{eq:Iris_Derivative}
{\rm{per}}\left( A \right) = {\left. {\frac{{\frac{{{d^S}}}{{dz_1^{\alpha _T^1} \ldots dz_t^{\alpha _T^t}}}\iota _t^A\left( {{z_1}, \ldots ,{z_t},\alpha} \right)}}{{\prod\limits_{i = 1}^t {\alpha _T^i} !}}} \right|_{{z_1} =  \ldots  = {z_t} = 0}}
\end{equation}  
where
   \begin{equation}\label{eq:alphaT}
\alpha _T^. = \sum\limits_{i = 1}^n {\alpha _i^.}
\end{equation} 
and
   \begin{equation}\label{eq:S}
S = \sum\limits_{i = 1}^t {\alpha _T^i}. 
\end{equation} 
 Since $\iota$ is entire, using the Cauchy's integral formula for positive radii $r_.$ around zero gives
  \begin{equation}\label{eq:Iris_DerivativeCauchy}
{\rm{per}}\left( A \right) = \frac{1}{{{{\left( {2\pi j} \right)}^t}}}\oint\limits_{\left| {{z_1}} \right| = {r_1}}  \ldots  \oint\limits_{\left| {{z_t}} \right| = {r_t}} {\frac{{\iota _t^A \left( {{z_1}, \ldots ,{z_t},\alpha} \right)}}{{{z_1}^{\alpha _T^1 + 1} \ldots {z_t}^{\alpha _T^t + 1}}}d{z_1} \cdots d{z_t}}. 
\end{equation}  
Equivalently, set $z_. = \exp \left( {j{\theta _.}}\right)$ for  $0\le \theta_.<2\pi $ to obtain

   \begin{equation}\label{eq:Iris_Integral}
{\rm{per}}\left( A \right) = \frac{1}{{{{\left( {2\pi } \right)}^t}}}\int\limits_0^{2\pi } { \cdots \int\limits_0^{2\pi } {\frac{{\prod\limits_{i = 1}^n {\left( {\sum\limits_{k = 1}^n {{A_{i,k}}\prod\limits_{l = 1}^t {\exp \left( {j\alpha _k^l{\theta _l}} \right)} } } \right)} }}{{\prod\limits_{l = 1}^t {\exp \left( {j\alpha _T^l{\theta _l}} \right)} }}} } d{\theta _1} \cdots d{\theta _t}.
\end{equation}

The generating function of the matrix permanent is an Iris function of order $n$ which uses $\alpha=I_n$, where $I_n$ denotes the identity matrix of size $n$. Clearly, the matrix $\alpha$ is of full rank, and in fact, Section 2 will show that any full rank matrix $\alpha \in \mathbb{Z^+}^{n\times n} $ satisfies \eqref{eq:Iris_Derivative}. However, a full rank matrix $\alpha$ also leads to the formulation in \eqref{eq:Iris_Integral} which involves $n$ nested integrals, and, up to the knowledge of the author, there exists no known efficient algorithm to evaluate \eqref{eq:Iris_Integral} for an arbitrary matrix $A$. 

The main contribution of this paper is to provide a new formulation for the permanent of an arbitrary $n\times n$ complex matrix via a second order Iris function that is defined in Theorem 1 in Section 2. The fixed reduced order formulation in Theorem 1 avoids the curse of dimensionality. However, the reduction in the rank of $\alpha$ is often accompanied with an exponential increase in its elements, as it is not easy to find a matrix $\alpha$ that satisfies \eqref{eq:Iris_Derivative} and whose elements increase at a polynomial rate with the dimension of the matrix, thus leading to a polynomial time algorithm. The elements of each row in the matrix $\alpha$ for the Iris function in Theorem 1 are selected to be consecutive prime numbers, thus they do not increase at an exponential rate as the prime number theorem demonstrates. Based on Theorem 1, Theorem 2 in Section 3 will give an alternative formulation of the permanent for matrices with complex or Gaussian integer elements. The alternative formulation in Theorem 2 is intended for digital machines to carry arithmetic operations. The method avoids the integral \eqref{eq:Iris_Integral} and requires basic arithmetic operations such as summation, multiplication and bit shift. Section 3 shows that on 0-1 or binary matrices, a loose upper bound on the complexity of exact evaluation of the permanent is $o\left( {{n^{27}}{{\left( {\log \left( {{n^3}} \right)} \right)}^6}{{\left( {{{\log }_2}\left( n \right)} \right)}^2}} \right)$ binary operations. But, since, \cite{valiant1979} defines the complexity class $\#P$, and shows that the permanent function of 0-1 matrix is $\#P$ complete, our result in Section 3 shows that the problems in the class $\#P$ can be solved in polynomial time.  Our final comments will be presented in the Conclusions section. 

\section{ Theorem 1 and its proof}
\begin{theorem}
Let $P$ be the set of prime numbers, and let $P_i$ denote the $i^{th}$ prime number. Define the $2\times n$ matrix $\alpha$
 \begin{equation}\label{eq:theorem1Alpha}
{\alpha _{i,j}} = \left\{ \begin{array}{l}
{P_{j + p}}{\mkern 1mu} {\mkern 1mu} {\mkern 1mu} {\mkern 1mu} {\mkern 1mu} {\mkern 1mu} {\mkern 1mu} {\mkern 1mu} {\mkern 1mu} {\mkern 1mu} {\mkern 1mu} {\mkern 1mu} {\mkern 1mu} {\mkern 1mu} {\mkern 1mu} {\mkern 1mu} {\mkern 1mu} {\mkern 1mu} {\mkern 1mu} {\mkern 1mu} {\mkern 1mu} {\mkern 1mu} {\mkern 1mu} {\mkern 1mu} {\mkern 1mu} {\mkern 1mu} {\mkern 1mu} {\mkern 1mu} {\mkern 1mu} {\mkern 1mu} {\mkern 1mu} {\mkern 1mu} {\mkern 1mu} {\mkern 1mu} \,\,\,\,\, if{\mkern 1mu} {\mkern 1mu} i = 1,{\mkern 1mu} {\mkern 1mu} j \in {Z_{0,n - 1}}\\
{\left( {{P_{j + p}}} \right)^2}{\mkern 1mu} {\mkern 1mu} {\mkern 1mu} {\mkern 1mu} {\mkern 1mu} {\mkern 1mu} {\mkern 1mu} {\mkern 1mu} {\mkern 1mu} {\mkern 1mu} {\mkern 1mu} {\mkern 1mu} {\mkern 1mu} {\mkern 1mu} {\mkern 1mu} {\mkern 1mu} {\mkern 1mu} {\mkern 1mu} {\mkern 1mu} {\mkern 1mu} {\mkern 1mu} {\mkern 1mu} {\mkern 1mu} {\mkern 1mu} {\mkern 1mu} if{\mkern 1mu} {\mkern 1mu} i = 2,{\mkern 1mu} {\mkern 1mu} j \in {Z_{0,n - 1}}
\end{array} \right.,
\end{equation}  
where $p\in \mathbb{Z}^+$ is selected such that
 \begin{equation}\label{eq:conditionP}
P_p> {n^2}\left( {1 + \frac{{{\Delta _{\max }}}}{{{P_p}}}} \right)\frac{{{\Delta _{\max }}}}{{{\Delta _{\min }}}},
\end{equation} 
where
 \begin{equation}\label{eq:deltaMax}
\Delta _{\max }=P_{p+n-1}-P_p,
\end{equation}
and
 \begin{equation}\label{eq:deltaMin}
{\Delta _{\min }} = \mathop {\arg \min }\limits_{i,j \in {Z_{p,p + n - 1}}} \left( {\left| {{P_i} - {P_j}} \right|} \right),
\end{equation}
 then, from \eqref{eq:Iris_Derivative} and \eqref{eq:Iris_Integral},
 \begin{equation}\label{eq:theorem1PermanentIris}
{\rm{per}}\left( A \right) = {\left. {\frac{{\frac{{{d^S}}}{{dz_1^{\alpha _T^1}dz_2^{\alpha _T^2}}}\iota _2^A\left( {{z_1},{z_2},\alpha } \right)}}{{\alpha _T^1!\alpha _T^2!}}} \right|_{{z_1} = {z_2} = 0}},
\end{equation}  
where $\alpha _T^.$ is defined in \eqref{eq:alphaT} and  $S$ is defined in \eqref{eq:S}. Equivalently,
 \begin{equation}\label{eq:theorem1PermanentExponential}
{\rm{per}}\left( A \right) = \frac{1}{{{{\left( {2\pi } \right)}^2}}}\int\limits_0^{2\pi } {\int\limits_0^{2\pi } {\frac{{\prod\limits_{i = 1}^n {\left( {\sum\limits_{k = 1}^n {{A_{i,k}}\exp \left( {\sum\limits_{t = 1}^2 {j\alpha _k^t{\theta _t}} } \right)} } \right)} }}{{\exp \left( {\sum\limits_{t = 1}^2 {j\alpha _T^t{\theta _t}} } \right)}}} } d{\theta _1}d{\theta _2}.
\end{equation}  
\end{theorem}

\subsection{Proof of Theorem 1}
\begin{proof}
Equivalence of \eqref{eq:theorem1PermanentIris} and \eqref{eq:theorem1PermanentExponential} follows directly from \eqref{eq:Iris_Derivative} to \eqref{eq:Iris_Integral}. We will prove that \eqref{eq:theorem1PermanentExponential} is true for the matrix $\alpha$ defined in \eqref{eq:theorem1Alpha}.
Write the integrand in \eqref{eq:Iris_Integral} as the ratio ${\rm P}\left( {{\theta _1}, \ldots ,{\theta _t}}  \right)$
\begin{equation}\label{eq:y0}
{\rm{P}}\left( {{\theta _1}, \ldots ,{\theta _t}} \right) = \frac{{f\left( {{\theta _1}, \ldots ,{\theta _t}} \right)}}{{g\left( {{\theta _1}, \ldots ,{\theta _t}} \right)}},
\end{equation}
where the numerator is defined as
\begin{equation}\label{eq:y1}
f\left( {{\theta _1}, \ldots ,{\theta _t}} \right) = \prod\limits_{i = 1}^n {\left( {\sum\limits_{k = 1}^n {{A_{i,k}}\exp \left( {\sum\limits_{l = 1}^t {j\alpha _k^l{\theta _l}} } \right)} } \right)} ,
\end{equation}
and the denominator is
 \begin{equation}\label{eq:y2}
g\left( {{\theta _1}, \ldots ,{\theta _t}} \right) = \exp \left( {\sum\limits_{l = 1}^t {j\alpha _T^l{\theta _l}} } \right).
 \end{equation}
 $f\left( {{\theta _1}, \ldots ,{\theta _t}} \right)$ is constructed such that each column of the matrix $A$ is modulated by a unique exponential, which, in fact is composed of the product of $t$ complex exponentials, i.e. $ {\exp \left( {\sum\limits_{l = 1}^t {j\alpha _k^l{\theta _l}} } \right)}$. Let the term  ${T_{i,k}}\left( {{\theta _1}, \ldots ,{\theta _t}} \right) = {A_{i,k}}\exp \left( {\sum\limits_{l = 1}^t {j\alpha _k^l{\theta _l}} } \right)$ correspond to the element at row $i$ and column $k$ in the modulated matrix $T{\left( {{\theta _1}, \ldots ,{\theta _t}} \right)}$, then  $f\left( {{\theta _1}, \ldots ,{\theta _t}} \right)$ can be rewritten as
  \begin{equation}\label{eq:y3}
f\left( {{\theta _1}, \ldots ,{\theta _t}} \right) = \prod\limits_{i = 1}^n {\sum\limits_{k = 1}^n {{T_{i,k}}\left( {{\theta _1}, \ldots ,{\theta _t}} \right)} } .
  \end{equation}
Changing the order of the summation and the product gives 
   \begin{equation}\label{eq:y4}
f\left( {{\theta _1}, \ldots ,{\theta _t}} \right) = \sum\limits_{\tau  \in {\Psi _n}} {\prod\limits_{i = 1}^n {{T_{i,{\tau _i}}}\left( {{\theta _1}, \ldots ,{\theta _t}} \right)} },
   \end{equation}
where $\Psi_n$ is the set of column vectors of dimension $n$ such that 
   \begin{equation}\label{eq:y5}
{\Psi _n} = \left\{ {\tau :\tau  = {{\left[ {\begin{array}{*{20}{c}}
{{\tau _1}}& \ldots &{{\tau _n}}
\end{array}} \right]}^T,}\,\,\tau_i\in {Z_{1,n}}\,\, \forall \,i \in {Z_{1,n}} } \right\}.
   \end{equation}    
Observe that there are $n^n$ distinct elements in $\Psi_n$. Each element corresponds to a distinct product of $n$ terms, i.e. the product of $n$ column terms, where each column term is selected from a unique row in the modulated matrix $T\left( {{\theta _1}, \ldots ,{\theta _t}} \right)$. $n!$ elements in $\Psi_n$ are made up from distinct column indices, i.e., their column indices form permutations of the first $n$ positive integers.  These elements make up the set ${S_n}$, defined in \eqref{eq:sn}. Clearly, ${S_n} \subset {\Psi _n}$. Define ${{\bar S}_n}$ to denote the set ${{\bar S}_n} = {\Psi _n}\backslash {S_n}$, thus ${S_n} \cap {\bar S_n} = \emptyset $ and ${S_n} \cup {\bar S_n} = {\Psi _n}$.

Define function  $h:{\Psi _n} \to \mathbb{C}$ 
   \begin{equation}\label{eq:yy1}
  h\left( \tau  \right) = \prod\limits_{i = 1}^n {{T_{i,{\tau _i}}}\left( {{\theta _1}, \ldots ,{\theta _t}} \right)}  = \prod\limits_{i = 1}^n {{A_{i,{\tau _i}}}\exp \left( {\sum\limits_{l = 1}^t {j\alpha _{{\tau _i}}^l{\theta _l}} } \right)}  = {h_A}^\tau {h_e}^\tau \left( {{\theta _1}, \ldots ,{\theta _t}} \right),
   \end{equation} 
where
   \begin{equation}\label{eq:yy11}
   {h_A}^{\tau} = \prod\limits_{i = 1}^n {{A_{i,\tau_i}}}, 
   \end{equation} 
and   
   \begin{equation}\label{eq:yy12}
   {h_e}^\tau \left( {{\theta _1}, \ldots ,{\theta _t}} \right) = \prod\limits_{i = 1}^n {\exp \left( {\sum\limits_{l = 1}^t {j\alpha _{\tau _i}^l{\theta _l}} } \right)} .
   \end{equation}   
Define the function $\varphi^.: \Psi_n \to \mathbb{Z}^+$ 
   \begin{equation}\label{eq:yy13}
   \varphi^. \left( \tau  \right) = \sum\limits_{i = 1}^n {{\alpha _{\tau_i}^.}},
   \end{equation}   
thus the function ${h_e}^{\tau}\left( {{\theta _1}, \ldots ,{\theta _t}} \right)$ can be expressed as
    \begin{equation}\label{eq:h_freq}
    {h_e}^\tau \left( {{\theta _1}, \ldots ,{\theta _t}} \right) = \exp \left( {\sum\limits_{i = 1}^t {j{\varphi ^i}\left( \tau  \right)} {\theta _i}} \right).
    \end{equation}
Define the function $\chi : Z_{1,n} \times {\Psi _n} \to Z_{0,n}$, such that, 
   \begin{equation}\label{eq:y6}
\chi \left( {i,\tau } \right) = \sum\limits_{j = 1}^n {{\text{I}_{\tau \left( j \right) = i}}}, 
   \end{equation} 
where $\text{I}_.$ is the indicator function.
Let the set $R_B$ be defined as follows,
\begin{equation}\label{eq:RB}
{R_B} = \left\{ {\beta :\beta  = {{\left[ {\begin{array}{*{20}{c}}
{{\beta _1}}& \ldots &{{\beta _n}}
\end{array}} \right]}^T},{\mkern 1mu} {\mkern 1mu} {\beta _i} \in {Z_{0,n}}{\mkern 1mu} {\mkern 1mu} \forall {\mkern 1mu} i \in {Z_{1,n}},\,\,\sum\limits_{i = 1}^n {{\beta _i}}  = n} \right\},
 \end{equation} 
define the function $B:{\Psi _n} \to R_B$ such that for
$\tau \in \Psi_n$  
\begin{equation}\label{eq:y7}
B\left( \tau  \right) = {\left[ {\begin{array}{*{20}{c}}
{{\beta _1}}& \ldots &{{\beta _n}}
\end{array}} \right]^T}
   \end{equation}
where ${\beta _i} = \chi \left( {i,\tau } \right)$ is the multiplicity 
of column $i$ in the vector $\tau$. It follows from \eqref{eq:y5} and \eqref{eq:y6} that, for all $\tau \in \Psi_n$
\begin{equation}\label{eq:y8}
{1_{1 \times n}}B\left( \tau  \right) = n, 
   \end{equation}
where ${1_{1 \times n}}$ denotes the $1\times n$ matrix of all ones. 
From  \eqref{eq:sn}, and since, permutations must have unique indices, for all $\tau \in S_n$
\begin{equation}\label{eq:n1}
B\left( \tau  \right) = {\left[ {\begin{array}{*{20}{c}}
1& \ldots &1
\end{array}} \right]^T}, \,\,\forall \tau \in S_n.
   \end{equation}
Since ${{\bar S}_n}$ excludes permutations,   
\begin{equation}\label{eq:condBSBar}
B\left( \tau  \right) \ne {\left[ {\begin{array}{*{20}{c}}
1& \ldots &1
\end{array}} \right]^T}, \,\,\forall {\tau  \in {{\bar S}_n}}.
   \end{equation}
Rewrite $\varphi^. \left( \tau  \right)$ in terms of  $B \left( \tau  \right)$ as
\begin{equation}\label{eq:n2}
{\varphi ^.}\left( \tau  \right) = {\alpha ^.}B\left( \tau  \right).
   \end{equation}
From \eqref{eq:n1} and \eqref{eq:n2}, $\forall \tau \in S_n$
\begin{equation}\label{eq:n3}
{\varphi ^.}\left( \tau  \right) = {\alpha ^.}{\left[ {\begin{array}{*{20}{c}}
1& \ldots &1
\end{array}} \right]^T} = \alpha _T^.
   \end{equation}   
Therefore, from \eqref{eq:y2} and \eqref{eq:yy12},  $\forall \tau \in S_n$
\begin{equation}\label{eq:n4}
{h_e}^\tau \left( {{\theta _1}, \ldots ,{\theta _t}} \right) = \exp \left( {\sum\limits_{i = 1}^t {j\alpha _T^i} {\theta _i}} \right) = g\left( {{\theta _1}, \ldots ,{\theta _t}} \right).
   \end{equation}  
In words, \eqref{eq:n4} shows that terms that correspond to the permanent in \eqref{eq:theorem1PermanentExponential} are free of the complex exponential. Rewrite \eqref{eq:theorem1PermanentExponential} to obtain
\begin{equation}\label{eq:n5}
\begin{array}{*{20}{l}}
{{\rm{per}}\left( A \right) = \frac{1}{{{{\left( {2\pi } \right)}^t}}}\int\limits_0^{2\pi }  \cdots  \int\limits_0^{2\pi } {\frac{{\prod\limits_{i = 1}^n {\left( {\sum\limits_{k = 1}^n {{A_{i,k}}\exp \left( {\sum\limits_{l = 1}^t {j\alpha _k^l{\theta _l}} } \right)} } \right)} }}{{\exp \left( {\sum\limits_{l = 1}^t {j\alpha _T^l{\theta _l}} } \right)}}} d{\theta _1} \ldots d{\theta _t}}\\
{ = \frac{1}{{{{\left( {2\pi } \right)}^t}}}\int\limits_0^{2\pi } { \cdots \int\limits_0^{2\pi } {\frac{{\sum\limits_{\tau  \in {S_n}} {\prod\limits_{i = 1}^n {{T_{i,{\tau _i}}}\left( {{\theta _1}, \ldots ,{\theta _t}} \right)} }  + \sum\limits_{\tau  \in {{\bar S}_n}} {\prod\limits_{i = 1}^n {{T_{i,{\tau _i}}}\left( {{\theta _1}, \ldots ,{\theta _t}} \right)} } }}{{g\left( {{\theta _1}, \ldots ,{\theta _t}} \right)}}} d{\theta _1} \ldots d{\theta _t}} }\\
\begin{array}{l}
= \frac{1}{{{{\left( {2\pi } \right)}^t}}}\int\limits_0^{2\pi } { \cdots \int\limits_0^{2\pi } {\left( {\sum\limits_{\tau  \in {S_n}} {\prod\limits_{i = 1}^n {{A_{i,{\tau _i}}}} }  + \frac{{\sum\limits_{\tau  \in {{\bar S}_n}} {\prod\limits_{i = 1}^n {{T_{i,{\tau _i}}}\left( {{\theta _1}, \ldots ,{\theta _t}} \right)} } }}{{g\left( {{\theta _1}, \ldots ,{\theta _t}} \right)}}} \right)} d{\theta _1} \ldots d{\theta _t}} \\
= {\rm{per}}\left( A \right) + \frac{1}{{{{\left( {2\pi } \right)}^t}}}\sum\limits_{\tau  \in {{\bar S}_n}} {{h_A}^\tau \int\limits_0^{2\pi } { \cdots \int\limits_0^{2\pi } {\frac{{{h_e}^\tau \left( {{\theta _1}, \ldots ,{\theta _t}} \right)}}{{g\left( {{\theta _1}, \ldots ,{\theta _t}} \right)}}d{\theta _1} \ldots d{\theta _t}} } } 
\end{array}\\
{ = {\rm{per}}\left( A \right) + {\rm{nuisance}},}
\end{array}
   \end{equation}  
where the term nuisance is defined as
   \begin{equation}\label{eq:n6}
\frac{1}{{{{\left( {2\pi } \right)}^t}}}\sum\limits_{\tau  \in {{\bar S}_n}} {{h_A}^\tau \int\limits_0^{2\pi } { \cdots \int\limits_0^{2\pi } {\frac{{{h_e}^\tau \left( {{\theta _1}, \ldots ,{\theta _t}} \right)}}{{g\left( {{\theta _1}, \ldots ,{\theta _t}} \right)}}d{\theta _1} \ldots d{\theta _t}} } } .
      \end{equation}  
Recall from \eqref{eq:y2} and \eqref{eq:yy12} that, both ${h_e}^\tau \left( \theta  \right)$ and $g \left( \theta  \right)$ are complex exponentials. Then, the integral inside \eqref{eq:n6} is
   \begin{equation}\label{eq:k1}
\int\limits_0^{2\pi } { \cdots \int\limits_0^{2\pi } {\frac{{{h_e}^\tau \left( {{\theta _1}, \ldots ,{\theta _t}} \right)}}{{g\left( {{\theta _1}, \ldots ,{\theta _t}} \right)}}d{\theta _1} \ldots d{\theta _t}} }  = \left\{ {\begin{array}{*{20}{c}}
0&{if{\mkern 1mu} {\mkern 1mu} {\mkern 1mu} {h_e}^\tau \left( {{\theta _1}, \ldots ,{\theta _t}} \right) \ne g\left( {{\theta _1}, \ldots ,{\theta _t}} \right){\mkern 1mu} }\\
{{{\left( {2\pi } \right)}^t}}&{otherwise}
\end{array}.} \right.
      \end{equation} 
To prove Theorem 1, we must show that there exists no  ${\tau  \in {{\bar S}_n}}$ such that, ${{h_e}^\tau \left( {{\theta _1}, \ldots ,{\theta _t}} \right) = g\left( {{\theta _1}, \ldots ,{\theta _t}} \right)}$  with the matrix $\alpha$ defined in \eqref{eq:theorem1Alpha} and $t=2$. The proof will use contradiction. First, assume that,  ${\exists\tau  \in {{\bar S}_n}}$, such that, ${{h_e}^\tau \left( {{\theta _1}, \ldots ,{\theta _t}} \right) = g\left( {{\theta _1}, \ldots ,{\theta _t}} \right)}$, then from  \eqref{eq:y2}, \eqref{eq:yy12} and \eqref{eq:n2} 
\begin{equation}\label{eq:k2}
\alpha B\left( \tau  \right) = {\left[ {\begin{array}{*{20}{c}}
{\alpha _T^1}& \cdots &{\alpha _T^t}
\end{array}} \right]^T}.
      \end{equation} 
Rewrite \eqref{eq:k2} as
\begin{equation}\label{eq:k2x}
\alpha x = C,
      \end{equation} 
where the column vector $x\in \mathbb{R}^n$ and $C={\left[ {\begin{array}{*{20}{c}}
{\alpha _T^1}& \cdots &{\alpha _T^t}
\end{array}} \right]^T}$. Let $t=n$, and let the $n\times n$ matrix $\alpha$ be of full rank, then, \eqref{eq:k2x} must have a unique solution $x$. From \eqref{eq:n3}, this solution is $x = {\left[ {\begin{array}{*{20}{c}}
1& \cdots &1
\end{array}} \right]^T}$, and because of \eqref{eq:condBSBar},  \eqref{eq:Iris_Integral} is satisfied for any matrix $\alpha \in \mathbb{Z^+}^{n\times n}$ that is of full rank. Clearly, the permanent generating function is a member of this class. 
 
The $2 \times n$ matrix $\alpha$ defined in \eqref{eq:theorem1Alpha} is, on the other hand, of rank  $\min \left( {2,n} \right)$. Therefore, for $n\le2$, $x = {\left[ {\begin{array}{*{20}{c}}
1& \cdots &1
\end{array}} \right]^T}$ must be the unique solution to \eqref{eq:k2x}. However, for $n>2$, \eqref{eq:k2x}  has infinitely many solutions.  For this case, we must show that if $x\ne {\left[ {\begin{array}{*{20}{c}}
1& \cdots &1
\end{array}} \right]^T}$ is a solution to \eqref{eq:k2x}, then $x \notin {R_B}$. Assume that $x\in R_B$ is a solution to \eqref{eq:k2x} and $x\ne {\left[ {\begin{array}{*{20}{c}}
1& \cdots &1
\end{array}} \right]^T}$, then, 
\begin{equation}\label{eq:nullSpace}
\alpha y = {0_{2 \times 1}},
\end{equation}
where $y=x-1_{n \times 1}$ and ${0_{i \times j}}$ denotes the $i \times j$ matrix of all zeros. Define
\begin{equation}\label{eq:RBminus}
R_B^ -  = \left\{ {\beta :\beta  = {{\left[ {\begin{array}{*{20}{c}}
				{{\beta _1}}& \ldots &{{\beta _n}}
				\end{array}} \right]}^T},{\beta _i} \in {Z_{ - 1,n - 1}}\,\forall i \in {Z_{1,n}},{\mkern 1mu} } \right\},
\end{equation}
then, to prove Theorem 1, it suffices to show that $R_B^ -  \cap \ker \left( \alpha  \right) = 0_{n \times 1}$.
From \eqref{eq:theorem1Alpha}, write the matrix $\alpha$ in the reduced row echelon form $\alpha_{rref}$
\begin{equation}\label{eq:rrEchelon}
\begin{array}{l}
\alpha  = \left[ {\begin{array}{*{20}{c}}
	{\alpha _1^1}&{\alpha _2^1}&{\alpha _3^1}& \cdots &{\alpha _n^1}\\
	{{{\left( {\alpha _1^1} \right)}^2}}&{{{\left( {\alpha _2^1} \right)}^2}}&{{{\left( {\alpha _3^1} \right)}^2}}& \cdots &{{{\left( {\alpha _n^1} \right)}^2}}
	\end{array}} \right]\\
\,\,\,\,\,\,\,\,\,\,\,\,\,\,\,\,\,\,\,\,\,\,\,\,\,\,\,\,\,\,\,\,\,\,\,\,\,\,\,\,\,\,\,\,\,\,\,\,\,\,\,\,\,\,\, \downarrow \\
\,\,\,\,\,\,\,\,\,\,\,\left[ {\begin{array}{*{20}{c}}
	1&{\frac{{\alpha _2^1}}{{\alpha _1^1}}}&{\frac{{\alpha _3^1}}{{\alpha _1^1}}}& \cdots &{\frac{{\alpha _n^1}}{{\alpha _1^1}}}\\
	1&{{{\left( {\frac{{\alpha _2^1}}{{\alpha _1^1}}} \right)}^2}}&{{{\left( {\frac{{\alpha _3^1}}{{\alpha _1^1}}} \right)}^2}}& \cdots &{{{\left( {\frac{{\alpha _n^1}}{{\alpha _1^1}}} \right)}^2}}
	\end{array}} \right]\\
\,\,\,\,\,\,\,\,\,\,\,\,\,\,\,\,\,\,\,\,\,\,\,\,\,\,\,\,\,\,\,\,\,\,\,\,\,\,\,\,\,\,\,\,\,\,\,\,\,\,\,\,\,\,\, \downarrow \\
\,\,\,\,\,\,\,\,\,\,\,\left[ {\begin{array}{*{20}{c}}
	1&{\frac{{\alpha _2^1}}{{\alpha _1^1}}}&{\frac{{\alpha _3^1}}{{\alpha _1^1}}}& \cdots &{\frac{{\alpha _n^1}}{{\alpha _1^1}}}\\
	0&{\frac{{{{\left( {\alpha _2^1} \right)}^2} - \alpha _1^1\alpha _2^1}}{{{{\left( {\alpha _1^1} \right)}^2}}}}&{\frac{{{{\left( {\alpha _3^1} \right)}^2} - \alpha _1^1\alpha _3^1}}{{{{\left( {\alpha _1^1} \right)}^2}}}}& \cdots &{\frac{{{{\left( {\alpha _n^1} \right)}^2} - \alpha _1^1\alpha _n^1}}{{{{\left( {\alpha _1^1} \right)}^2}}}}
	\end{array}} \right]\\
\,\,\,\,\,\,\,\,\,\,\,\,\,\,\,\,\,\,\,\,\,\,\,\,\,\,\,\,\,\,\,\,\,\,\,\,\,\,\,\,\,\,\,\,\,\,\,\,\,\,\,\,\,\,\, \downarrow \\
\,\,\,\,\,\,\,\,\,\,\,\left[ {\begin{array}{*{20}{c}}
	1&{\frac{{\alpha _2^1}}{{\alpha _1^1}}}&{\frac{{\alpha _3^1}}{{\alpha _1^1}}}& \cdots &{\frac{{\alpha _n^1}}{{\alpha _1^1}}}\\
	0&1&{\frac{{{{\left( {\alpha _3^1} \right)}^2} - \alpha _1^1\alpha _3^1}}{{{{\left( {\alpha _2^1} \right)}^2} - \alpha _1^1\alpha _2^1}}}& \cdots &{\frac{{{{\left( {\alpha _n^1} \right)}^2} - \alpha _1^1\alpha _n^1}}{{{{\left( {\alpha _2^1} \right)}^2} - \alpha _1^1\alpha _2^1}}}
	\end{array}} \right]\\
\,\,\,\,\,\,\,\,\,\,\,\,\,\,\,\,\,\,\,\,\,\,\,\,\,\,\,\,\,\,\,\,\,\,\,\,\,\,\,\,\,\,\,\,\,\,\,\,\,\,\,\,\,\,\, \downarrow \\
\,\,\,\,\,\,\,\,\,\,\left[ {\begin{array}{*{20}{c}}
	1&0&{\frac{{\alpha _3^1}}{{\alpha _1^1}}\frac{{\alpha _2^1 - \alpha _3^1}}{{\alpha _2^1 - \alpha _1^1}}}& \cdots &{\frac{{\alpha _n^1}}{{\alpha _1^1}}\frac{{\alpha _2^1 - \alpha _n^1}}{{\alpha _2^1 - \alpha _1^1}}}\\
	0&1&{\frac{{\alpha _3^1}}{{\alpha _2^1}}\frac{{\alpha _3^1 - \alpha _1^1}}{{\alpha _2^1 - \alpha _1^1}}}& \cdots &{\frac{{\alpha _n^1}}{{\alpha _2^1}}\frac{{\alpha _n^1 - \alpha _1^1}}{{\alpha _2^1 - \alpha _1^1}}}
	\end{array}} \right]=\alpha_{rref}.
\end{array}
\end{equation}
From \eqref{eq:rrEchelon}, the rank of $\alpha$ is 2 and the nullity of $\alpha$ is $n-2$. The kernel of $\alpha$ is spanned by the set of $n-2$ column vectors $\left\{ {{V^1}, \ldots ,{V^{n - 2}}} \right\}$, where, $\forall i \in {Z_{1,n}}$ the column vector ${V^i} = {\left[ {\begin{array}{*{20}{c}}
		{V_1^i}& \ldots &{V_n^i}
		\end{array}} \right]^T} \in \mathbb{R}^n$ can be written as 

\iffalse
\begin{equation}\label{eq:kernelBasis}
V_j^i = \left\{ \begin{array}{l}
\alpha _{i + 2}^1\alpha _2^1\left( {\alpha _{i + 2}^1 - \alpha _2^1} \right)\,\,\,\,\,\,\,\,\,\,if\,\,\,j = 1\\
\alpha _1^1\alpha _{i + 2}^1\left( {\alpha _1^1 - \alpha _{i + 2}^1} \right)\,\,\,\,\,\,\,\,\,\,if\,\,\,j = 2\\
\alpha _1^1\alpha _2^1\left( {\alpha _2^1 - \alpha _1^1} \right)\,\,\,\,\,\,\,\,\,\,\,\,\,\,\,\,\,\,\,\,if\,\,\,j = i + 2\\
0\,\,\,\,\,\,\,\,\,\,\,\,\,\,\,\,\,\,\,\,\,\,\,\,\,\,\,\,\,\,\,\,\,\,\,\,\,\,\,\,\,\,\,\,\,\,\,\,\,\,\,\,\,\,\,\,\,\,if\,\,\,j \in {Z_{1,n}}\backslash \left\{ {1,2,i + 2} \right\}
\end{array} \right..
\end{equation}
\fi

\begin{equation}\label{eq:kernelBasis}
V_j^i = \left\{ {\begin{array}{*{20}{l}}
	{\frac{{\alpha _{i + 2}^1\left( {\alpha _{i + 2}^1 - \alpha _2^1} \right)}}{{\alpha _1^1\left( {\alpha _2^1 - \alpha _1^1} \right)}}{\mkern 1mu} {\mkern 1mu} {\mkern 1mu} {\mkern 1mu} {\mkern 1mu} {\mkern 1mu} {\mkern 1mu} {\mkern 1mu} {\mkern 1mu} {\mkern 1mu} if{\mkern 1mu} {\mkern 1mu} {\mkern 1mu} j = 1}\\
	{\frac{{\alpha _{i + 2}^1\left( {\alpha _1^1 - \alpha _{i + 2}^1} \right)}}{{\alpha _2^1\left( {\alpha _2^1 - \alpha _1^1} \right)}}{\mkern 1mu} {\mkern 1mu} {\mkern 1mu} {\mkern 1mu} {\mkern 1mu} {\mkern 1mu} {\mkern 1mu} {\mkern 1mu} {\mkern 1mu} {\mkern 1mu} if{\mkern 1mu} {\mkern 1mu} {\mkern 1mu} j = 2}\\
	{1{\mkern 1mu} {\mkern 1mu} {\mkern 1mu} {\mkern 1mu} {\mkern 1mu} {\mkern 1mu} {\mkern 1mu} {\mkern 1mu} {\mkern 1mu} {\mkern 1mu} {\mkern 1mu} {\mkern 1mu} {\mkern 1mu} {\mkern 1mu} {\mkern 1mu} {\mkern 1mu} {\mkern 1mu} {\mkern 1mu} {\mkern 1mu} {\mkern 1mu} \,\,\,\,\,\,\,\,\,\,\,\,\,\,\,\,\,\,\,\,\,\,\,\,\,\,\,\,\,if{\mkern 1mu} {\mkern 1mu} {\mkern 1mu} j = i + 2}\\
	{0\,\,\,\,\,\,\,\,\,\,\,\,\,\,\,\,\,\,\,\,\,\,\,\,\,\,\,\,\,\,\,\,\,\,\,if{\mkern 1mu} {\mkern 1mu} {\mkern 1mu} j \in {Z_{1,n}}\backslash \left\{ {1,2,i + 2} \right\}}
	\end{array}} \right..
\end{equation}
Write the $n\times n-2$ matrix $N = \left[ {{V^1}, \ldots ,{V^{n - 2}}} \right]$,
\begin{equation}\label{eq:Nmatrix}
\begin{array}{l}
N = \left[ {\begin{array}{*{20}{c}}
	{\frac{{\alpha _3^1\left( {\alpha _3^1 - \alpha _2^1} \right)}}{{\alpha _1^1\left( {\alpha _2^1 - \alpha _1^1} \right)}}}&{\frac{{\alpha _4^1\left( {\alpha _4^1 - \alpha _2^1} \right)}}{{\alpha _1^1\left( {\alpha _2^1 - \alpha _1^1} \right)}}}& \cdots &{\frac{{\alpha _{n - 1}^1\left( {\alpha _{n - 1}^1 - \alpha _2^1} \right)}}{{\alpha _1^1\left( {\alpha _2^1 - \alpha _1^1} \right)}}}&{\frac{{\alpha _n^1\left( {\alpha _n^1 - \alpha _2^1} \right)}}{{\alpha _1^1\left( {\alpha _2^1 - \alpha _1^1} \right)}}}\\
	{\frac{{\alpha _3^1\left( {\alpha _1^1 - \alpha _3^1} \right)}}{{\alpha _2^1\left( {\alpha _2^1 - \alpha _1^1} \right)}}}&{\frac{{\alpha _4^1\left( {\alpha _1^1 - \alpha _4^1} \right)}}{{\alpha _2^1\left( {\alpha _2^1 - \alpha _1^1} \right)}}}& \cdots &{\frac{{\alpha _{n - 1}^1\left( {\alpha _1^1 - \alpha _{n - 1}^1} \right)}}{{\alpha _2^1\left( {\alpha _2^1 - \alpha _1^1} \right)}}}&{\frac{{\alpha _n^1\left( {\alpha _1^1 - \alpha _n^1} \right)}}{{\alpha _2^1\left( {\alpha _2^1 - \alpha _1^1} \right)}}}\\
	1&0& \cdots &0&0\\
	0&1& \cdots &0&0\\
	0&0& \cdots &0&0\\
	\vdots & \vdots & \cdots & \vdots & \vdots \\
	0&0& \cdots &0&0\\
	0&0& \cdots &1&0\\
	0&0& \cdots &0&1
	\end{array}} \right]\\
= \left[ {\begin{array}{*{20}{c}}
	{{N_{UP}}}\\
	{{I_{n - 2}}}
	\end{array}} \right],
\end{array}
\end{equation}
where $N_{UP}$ is the upper two rows of the matrix $N$.

Let $S\in \mathbb{R}^n$ be a linear combination of the columns of matrix $N$, i.e. the vectors that span the kernel of $\alpha$, 
\begin{equation}\label{eq:linComb}
 S = \sum\limits_{i = 1}^{n - 2} {{\gamma _i}{V^i}}.
\end{equation}
where $\gamma_. \in \mathbb{R}$. From \eqref{eq:RBminus} and  \eqref{eq:Nmatrix}, if there exists a vector $S$ which is both $S\in R_B^ -$ and $S\in  \cap \ker \left( \alpha  \right)$, then $\gamma_. \in Z_{-1,n-1}$.

Recall from \eqref{eq:theorem1Alpha} that the elements of $\alpha^1$ are distinct prime numbers. Therefore, if there exists a vector $S$ with integer elements which is both $S\in R_B^ -$ and $S\in  \cap \ker \left( \alpha  \right)$, the following two conditions must jointly be met
\begin{equation}\label{eq:condition12}
\begin{array}{l}
{\text{Condition I. }}\,\,\,\,\,{N_{UP}}^1\gamma  \equiv 0 \,\,\,\left( {\bmod \, \alpha _1^1} \right) \\
{\text{Condition II. }}\,\,\,{N_{UP}}^2\gamma  \equiv 0 \,\,\,\left( {\bmod \, \alpha _2^1} \right)
\end{array}.
\end{equation}
where $N_{UP}^1$ is the first and $N_{UP}^2$ is the second row of the $2 \times n-2$ matrix $N_{UP}$ and the column vector $\gamma  = {\left[ {\begin{array}{*{20}{c}}
		{{\gamma _1}}& \ldots &{{\gamma _{n - 2}}}
		\end{array}} \right]^T}$. 
Let $\gamma_. \in Z_{-1,n-1}$, then from \eqref{eq:theorem1Alpha} and \eqref{eq:Nmatrix}, an upper bound for ${N_{UP}}^.\gamma $ is given by
\begin{equation}\label{eq:condPrestated}
{N_{UP}}^.\gamma  < {n^2}\left( {1 + \frac{{{\Delta _{\max }}}}{{{P_p}}}} \right)\frac{{{\Delta _{\max }}}}{{{\Delta _{\min }}}}.
\end{equation}
With the restriction \eqref{eq:conditionP}, Condition I and Condition can be rewritten as
\begin{equation}\label{eq:condition12exact}
\begin{array}{l}
{\text{Condition I restricted. }}\,\,\,\,\,{N_{UP}}^1\gamma  = 0 \\
{\text{Condition II restricted. }}\,\,\,{N_{UP}}^2\gamma  = 0 
\end{array}.
\end{equation}
Therefore the first two elements of the vector $S$, i.e. $S_1$ and $S_2$, must be zero. From \eqref{eq:Nmatrix}, Condition I restricted requires
\begin{equation}\label{eq:condition1exact}
\sum\limits_{i = 1}^{n - 2} {\alpha _{i + 2}^1\left( {\alpha _{i + 2}^1 - \alpha _2^1} \right)} {\gamma _i} = 0, 
\end{equation}
and  Condition II restricted requires
\begin{equation}\label{eq:condition2exact}
\sum\limits_{i = 1}^{n - 2} {\alpha _{i + 2}^1\left( {\alpha _1^1 - \alpha _{i + 2}^1} \right)} {\gamma _i} = 0.
\end{equation}
From \eqref{eq:condition1exact} and \eqref{eq:condition2exact}, Condition I restricted and Condition II restricted jointly require
\begin{equation}\label{eq:JointCondition}
{\alpha _{red}}\gamma  = 0_{2\times 1},
\end{equation}
where 
\begin{equation}\label{eq:alphaRed}
{\alpha _{red}} = \left[ {\begin{array}{*{20}{c}}
	{\alpha _3^1}& \cdots &{\alpha _n^1}\\
	{{{\left( {\alpha _3^1} \right)}^2}}& \cdots &{{{\left( {\alpha _n^1} \right)}^2}}
	\end{array}} \right]
\end{equation}
is the matrix $\alpha$ with its first two columns reduced and  $\gamma_. \in Z_{-1,n-1}$. But, since, 	
\begin{equation}\label{eq:kerDual}	
	\ker \left( {\left[ {\begin{array}{*{20}{c}}
		{\alpha _i^1}&{\alpha _j^1}\\
		{{{\left( {\alpha _i^1} \right)}^2}}&{{{\left( {\alpha _j^1} \right)}^2}}
		\end{array}} \right]} \right) = 0_{2\times1}
\end{equation}
for $i\ne j$, induction gives $R_B^ -  \cap \ker \left( \alpha  \right) = 0_{n \times 1}$,  therefore, $x= {\left[ {\begin{array}{*{20}{c}}
		1& \cdots &1
		\end{array}} \right]^T}$ is the only solution to \eqref{eq:k2x} in $R_B$.
\end{proof}

\section{Theorem 2, its proof and a polynomial time implementation on 0-1 matrices}
Let $k\left(i\right)$ denote the set of complex or 
``Gaussian" integers, and for $M\in \mathbb{Z}^+$, let ${A_n^M}$ to be the set of $n\times n$ matrices with bounded complex integer elements such that,
\begin{equation}\label{eq:Mn}
A_n^M = \left\{ \begin{array}{l}
A:A \in k{\left( i \right)^{n \times n}},\\
\left| {{A_{k,l}}} \right|{\mkern 1mu}  < M\,\,\,\forall k \in {Z_{1,n}}\,\,\text{and}\,\,\,\forall l \in {Z_{1,n}}\,\,\,\text{and}\,M \in \mathbb{Z}^+\,
\end{array} \right\},
\end{equation}
given a row vector $\alpha = \left[ {\begin{array}{*{20}{c}}
 {{\alpha _1}}& \ldots &{{\alpha _n}}
 \end{array}} \right]$, $\alpha \in \left({\mathbb{Z}^+}\right)^{1\times n}$,  such that $x= {\left[ {\begin{array}{*{20}{c}}
  1& \cdots &1
  \end{array}} \right]^T}$ is the unique solution to \eqref{eq:k2x}  $\forall x\in R_B$, i.e. ${\iota _1^A\left( {z,\alpha } \right)}$ satisfies \eqref{eq:Iris_Derivative} for an arbitrary matrix $A \in A^M_n$, define the function $\text{per}^A_m: \mathbb{C} \times \left({\mathbb{Z}^+}\right)^{1\times n} \to k\left(i\right)$ as 
 \begin{equation}\label{eq:Perm}
{\rm{per}}_m^A\left( {z,\alpha } \right) = \left[\frac{1}{{{z^{{\alpha _T}}}}} {\iota _1^A\left( {z,\alpha } \right)} \right] = \left[ {\frac{1}{{{z^{{\alpha _T}}}}}\prod\limits_{i = 1}^n {\sum\limits_{k = 1}^n {{A_{i,k}}{z^{{\alpha _k}}}} } } \right]
\end{equation}
where $z\in \mathbb{Z}^+$, $\alpha_T$ is defined in \eqref{eq:alphaT}  and $\left[.\right]$ denotes  the round to the nearest integer operation.

\begin{theorem}
Let $z\in \mathbb{Z}^+$ and $z > 2\left(Mn\right)^n$, permanent of $A$ is congruent to ${\rm{per}}_m^A\left( {z,\alpha } \right)$ modulus $z$, i.e.
\begin{equation}\label{eq:thrm2_1}
{\rm{per}}\left( A \right) \equiv {\rm{per}}_m^A\left( {z,\alpha } \right)\,\,\,\left( {\bmod \,z} \right),
\end{equation}
and for $a\in Z_{0,z-1}$ and $b\in Z_{0,z-1}$, let $a + bj$ denote  the least residue of ${\rm{per}}_m^A\left( {z,\alpha } \right)$ modulus $z$, then   
\begin{equation}\label{eq:theorem2Correction}
{\rm{per}}\left( A \right) = a + bj - z({{\rm{I}}_{a > M^nn!}} + j{{\rm{I}}_{b >  M^nn!}}).
\end{equation}
\end{theorem}
\begin{proof}
Theorem 2 stems directly from the proof of Theorem 1 with the condition that $x= {\left[ {\begin{array}{*{20}{c}}
  1& \cdots &1
  \end{array}} \right]^T}$ is the unique solution to \eqref{eq:k2x}  $\forall x\in R_B$. Rewrite the Iris function 
\begin{equation}\label{eq:theorem2Iris}
\iota _1^A\left( {z,\alpha } \right) = \prod\limits_{i = 1}^n {\sum\limits_{k = 1}^n {\left( {{A_{i,k}}{z^{{\alpha _k}}}} \right)} }  = \sum\limits_{\tau  \in {\Psi _n}} {{z^{h_e^z\left( \tau  \right)}}\prod\limits_{i = 1}^n {{A_{i,{\tau _i}}}} }  
\end{equation}  
where the function $h_e^z: \Psi _n \to \mathbb{Z}^+$ is defined as 
\begin{equation}\label{eq:theorem2he}
h_e^z\left( \tau  \right) = \sum\limits_{i = 1}^n {{\alpha _{{\tau _i}}}} .
\end{equation}
Then,
\begin{equation}\label{eq:theorem2Cond}
\begin{array}{l}
\frac{1}{{{z^{{\alpha _T}}}}}\iota _1^A\left( {z,\alpha } \right) = \frac{1}{{{z^{{\alpha _T}}}}}\sum\limits_{\tau  \in {S_n}} {{z^{h_e^z\left( \tau  \right)}}\prod\limits_{i = 1}^n {{A_{i,{\tau _i}}}}  + } \frac{1}{{{z^{{\alpha _T}}}}}\sum\limits_{\tau  \in {{\bar S}_n}} {{z^{h_e^z\left( \tau  \right)}}\prod\limits_{i = 1}^n {{A_{i,{\tau _i}}}} } \\
 = {\rm{per}}\left( A \right) + \sum\limits_{\tau  \in {{\bar S}_n}} {{z^{h_e^z\left( \tau  \right) - {\alpha _T}}}\prod\limits_{i = 1}^n {{A_{i,{\tau _i}}}} }. 
\end{array}
\end{equation}
Partition ${{\bar S}_n}$ into three disjoint sets ${{\bar S}_n}^+$,
 ${{\bar S}_n}^0$ and ${{\bar S}_n}^-$ such that, 
\begin{equation}\label{eq:theorem2Sbarsets}
\begin{array}{l}
{{\bar S}_n}^ +  = \left\{ {\tau :\tau  \in {{\bar S}_n},\,h_e^z\left( \tau  \right) > {\alpha _T}\,} \right\}\\
{{\bar S}_n}^0 = \left\{ {\tau :\tau  \in {{\bar S}_n},\,h_e^z\left( \tau  \right) = {\alpha _T}\,} \right\}\\
{{\bar S}_n}^ -  = \left\{ {\tau :\tau  \in {{\bar S}_n},\,h_e^z\left( \tau  \right) < {\alpha _T}\,} \right\}.
\end{array}
\end{equation}
Since  $x= {\left[ {\begin{array}{*{20}{c}}
  1& \cdots &1
  \end{array}} \right]^T}$ is the unique solution to \eqref{eq:k2x}  $\forall x\in R_B$, ${{\bar S}_n}^0 =\emptyset$. Rewrite \eqref{eq:theorem2Cond} 
  \begin{equation}\label{eq:theorem2CondRewritten}
\frac{1}{{{z^{{\alpha _T}}}}}\iota _1^A\left( {z,\alpha } \right) = {\rm{per}}\left( A \right) + {F^ + } + {F^ - }
\end{equation}
where
  \begin{equation}\label{eq:theorem2Fplus}
{{F^ + } = \sum\limits_{\tau  \in {{\bar S}_n}^ + } {{z^{h_e^z\left( \tau  \right) - {\alpha _T}}}\prod\limits_{i = 1}^n {{A_{i,{\tau _i}}}} } }
\end{equation}
and
  \begin{equation}\label{eq:theorem2Fminus}
{F^ - } = \sum\limits_{\tau  \in {{\bar S}_n}^ - } {{z^{h_e^z\left( \tau  \right) - {\alpha _T}}}\prod\limits_{i = 1}^n {{A_{i,{\tau _i}}}} } 
\end{equation}
Clearly, ${F^ + } \equiv 0\,{\mkern 1mu} {\mkern 1mu} {\mkern 1mu} \left( {\,\bmod \,{\mkern 1mu} z} \right),$ and ${F^ - } \le \sum\limits_{\tau  \in {{\bar S}_n}^ - } {{z^{h_e^z\left( \tau  \right) - {\alpha _T}}}\left| {\prod\limits_{i = 1}^n {{A_{i,{\tau _i}}}} } \right|}$, thus ${F^ - } \le \sum\limits_{\tau  \in {{\bar S}_n}} {{z^{ - 1}}\left| {\prod\limits_{i = 1}^n {{A_{i,{\tau _i}}}} } \right|} $ for $z>1$. Since elements of matrix $A$ are, by definition, bounded by $M$,  $\sum\limits_{\tau  \in {{\bar S}_n}} {\prod\limits_{i = 1}^n {{A_{i,{\tau _i}}}} }  < {\left( {Mn} \right)^n}$. Then, for $z > 2\left(Mn\right)^n$,  ${F^ - } <0.5$. Therefore, alias terms in ${\rm{per}}_m^A\left( {z,\alpha } \right)$ that are due to ${{\bar S}_n}^-$ and rounded down to zero, and alias terms due to ${{\bar S}_n}^+$ are congruent to zero $\left(\bmod \,z\right)$. This completes the first part of the proof. For the second part, since $\left| {{\rm{per}}\left( A \right)} \right| \le  M^nn! < \frac{z}{2}$, if, any part, real or complex, of the least residue of ${\rm{per}}_m^A\left( {z,\alpha } \right)$ is greater than $ M^nn!$, the result must be negative, and since $  M^nn! < \frac{z}{2}$, there cannot be an overflow. \eqref{eq:theorem2Correction} corrects for the offset in the least residue, if any part of the result is negative.
\end{proof}
\subsection{Polynomial time computational complexity for zero-one matrices}
In this section, we will consider an implementation of the method in Theorem 2 for an arbitrary $n \times n$ zero-one matrix on a digital computer that is capable of carrying out fixed point binary arithmetic operations. The emphasis here is not to provide an efficient implementation, but to demonstrate that the permanent of a 0-1 matrix can be calculated with a polynomial time computational complexity. 

To be compatible with the method defined in Theorem 2, Lemma 1 shows that the formulation of the matrix permanent via the second order Iris function defined in Theorem 1 can be equivalently represented with a first order Iris function.    

\begin{lemma}
	Define the $1\times n$ matrix $\alpha$
	\begin{equation}\label{eq:lemma1}
	{\alpha _j} = {P_{j + p}} + \beta{\left( {{P_{j + p}}} \right)^2},
	\end{equation}
	where $p\in \mathbb{Z}^+$ is selected according to \eqref{eq:conditionP}, and $\beta \in \mathbb{Z}^+$ such that 
	\begin{equation}\label{eq:lemma1Beta}
	\beta > n{P_{p + n - 1}}
	\end{equation}
	for $j\ne k \in Z_{0,n-1}$, then ${\iota _1^A\left( {z,\alpha } \right)}$ satisfies \eqref{eq:Iris_Derivative} for an arbitrary $n \times n$ complex matrix $A$.
	\begin{proof}
		From \eqref{eq:theorem1Alpha}, \eqref{eq:lemma1} can be rewritten as
		\begin{equation}\label{eq:lemma1eqAlternative}
		\alpha=\alpha^1+\beta\alpha^2.
		\end{equation} Theorem 1 proves that \eqref{eq:k2x} has unique solution $x\in R_B$, that is the equations
		\begin{equation}\label{eq:lemma1eq1}
		{\alpha ^1}x = \alpha _T^1
		\end{equation}
		and 
		\begin{equation}\label{eq:lemma1eq2}
		{\alpha ^2}x = \alpha _T^2
		\end{equation}
		cannot be jointly satisfied for $x\in R_B$ and 
		$x\ne {\left[ {\begin{array}{*{20}{c}}
				1& \cdots &1
				\end{array}} \right]^T}$. From \eqref{eq:theorem1Alpha} and \eqref{eq:RB},
		\begin{equation}\label{eq:lemma1upperBound}		 
		{\alpha ^1}x \le n{P_{j + n - 1}}
		\end{equation}
		$\forall x\in R_B$. From \eqref{eq:lemma1} and \eqref{eq:lemma1Beta}, for $x\in R_B$, if $ {\alpha ^2}x \ne \alpha _T^2$ , then 
		\begin{equation}\label{eq:lemma1dist}
		\beta\left| {{\alpha ^2}x - \alpha _T^2} \right| >  n{P_{j + n - 1}}.
		\end{equation} Therefore, for $x\in R_B$ and $x\ne {\left[ {\begin{array}{*{20}{c}}
				1& \cdots &1
				\end{array}} \right]^T}$, if \eqref{eq:lemma1eq2} is false, from \eqref{eq:lemma1upperBound}	 and \eqref{eq:lemma1dist}, $\left( {{\alpha ^1} + \beta {\alpha ^2}} \right)x \ne \alpha _T^1 + \beta \alpha _T^2$ , but if \eqref{eq:lemma1eq2} is true, Theorem 1 proves that \eqref{eq:lemma1eq1} must be false, thus  $\left( {{\alpha ^1} + \beta {\alpha ^2}} \right)x \ne \alpha _T^1 + \beta \alpha _T^2$.	   
	\end{proof}
\end{lemma}

If we further assume that $z>2n^n$ is selected to be an integer power of two, i.e. $z=2^k$, $k\in \mathbb{Z}^+$, then the sum
\[{\sum\limits_{k = 1}^n {{A_{i,k}}{z^{{\alpha _k}}}} }\]
for each row in a binary matrix $A$ in \eqref{eq:Perm} corresponds to setting at most $n$ bits on a sparse binary number, and the division
\[{\frac{1}{{{z^{{\alpha _T}}}}}}\] corresponds to a right shift operation.Clearly, the computational complexity in \eqref{eq:Perm} will be dominated by the $n-1$ multiplication operations which may involve extremely large, but sparse, binary numbers. 

The matrix $\alpha$ in Theorem 2 is constructed with prime numbers whose asymptotic distribution is given by the prime number theorem. More specifically, the prime number theorem shows that the prime counting function $\pi \left( n \right)$ is asymptotic to $\frac{n}{{\log \left( n \right)}}$, i.e.   
\begin{equation}\label{eq:PrimeTheorem}
\pi \left( n \right) \sim \frac{n}{{\log \left( n \right)}}.
\end{equation}
\eqref{eq:PrimeTheorem} is equivalent to the statement that $n^{\text{th}}$ prime number is asymptotic to $n{\log \left( n \right)}$ (see \cite{hardyPrimes}, Theorem 8), i.e. 
\begin{equation}\label{eq:PrimeAsymptotic}
{P_n} \sim n\log \left( n \right).
\end{equation}
Set $p=n^3$, then $P_p \sim n^3 \log\left(n^3\right)$ and
\begin{equation}\label{eq:deltaMax2}
\begin{array}{l}
{\Delta _{\max }} = {P_{p + n - 1}} - {P_p} \sim \left( {{n^3} + n - 1} \right)\log \left( {{n^3} + n - 1} \right) - {n^3}\log \left( {{n^3}} \right)\\
\sim n\log \left( {{n^3}} \right)
\end{array}
\end{equation}
and
\begin{equation}\label{eq:deltaMin2}
\begin{array}{l}
{\Delta _{\min }} = {P_{p + 1}} - {P_p} \sim \left( {{n^3} + 1} \right)\log \left( {{n^3} + 1} \right) - {n^3}\log \left( {{n^3}} \right) \sim \log \left( {{n^3}} \right)
\end{array},
\end{equation}
thus \eqref{eq:conditionP} is asymptotically satisfied.
From \eqref{eq:lemma1} and \eqref{eq:lemma1Beta},
\begin{equation}\label{eq:alphaTLemma1}
{\alpha _.} < {P_{p + n - 1}} + \left(n+1\right){\left( {{P_{p + n - 1}}} \right)^3},
\end{equation}
and 
\begin{equation}\label{eq:alphaTLemma1Cont}
\begin{array}{l}
{P_{p + n - 1}} + n{\left( {{P_{p + n - 1}}} \right)^3} \sim \left( {{n^3} + n - 1} \right)\log \left( {{n^3} + n - 1} \right)\\
+ \left(n+1\right){\left( {\left( {{n^3} + n - 1} \right)\log \left( {{n^3} + n - 1} \right)} \right)^3} \sim {n^{10}}{\left( {\log \left( {{n^3}} \right)} \right)^3}
\end{array}.
\end{equation}
For $n\gg2$, set $z=2^k<n^{n+1}$ and from \eqref{eq:Perm} and \eqref{eq:alphaTLemma1Cont}
\begin{equation}\label{eq:PermLimit}
\prod\limits_{i = 1}^n {\sum\limits_{k = 1}^n {{A_{i,k}}{z^{{\alpha _k}}}} }  \prec {n^{{n^{13}}{{\left( {\log \left( {{n^3}} \right)} \right)}^3}}}.
\end{equation}
The complexity in the computation of \eqref{eq:Perm} is dominated by the $n-1$ multiplications and from \eqref{eq:PermLimit} the result of multiplications is $o\left( {{n^{{n^{13}}{{\left( {\log \left( {{n^3}} \right)} \right)}^3}}}} \right)$, thus requiring  $o\left( {{n^{13}}{{\left( {\log \left( {{n^3}} \right)} \right)}^3}{{\log }_2}\left( n \right)} \right)$ bits. A naive implementation of the multiplication of two ${n^{13}}{\left( {\log \left( {{n^3}} \right)} \right)^3}{\log _2}\left( n \right)$ bit numbers requires ${\left( {{n^{13}}{{\left( {\log \left( {{n^3}} \right)} \right)}^3}{{\log }_2}\left( n \right)} \right)^2}$ binary operations, therefore the computational complexity in the multiplication operations in \eqref{eq:Perm} is $o\left( {{n^{27}}{{\left( {\log \left( {{n^3}} \right)} \right)}^6}{{\left( {{{\log }_2}\left( n \right)} \right)}^2}} \right)$ binary operations.

\section{Conclusions}
The matrix permanent lies at the heart of many analytical combinatorics methods which can describe both enumeration and optimization problems. This paper provides two solutions to the matrix permanent problem. The first solution, valid for arbitrary complex matrices, gives a formulation of the matrix permanent as a nested integral involving complex exponentials. The second solution is based on the first solution, and is intended to be implemented on a digital machine using basic arithmetical operations on complex integers. The paper shows that the second method can compute the permanent of a 0-1 matrix in polynomial time. Both methods require machines that can operate on large numbers. Therefore, an important conclusion from our results is that problems related to the matrix permanent can be solved in polynomial time using a classical binary digital computing machine with large registers.

\section*{Acknowledgments}
The author would like to thank to Professor Roy Streit and Professor Murat Efe for many helpful discussions on this work and to Professor Thiagalingam Kirubarajan for supporting his visit to the Department of Electrical and Computer Engineering of McMaster University where the research leading to this manuscript began as a search to find an efficient approximation to the target process marginalization problem in Bayesian multi-target tracking filters. 

The author was an Assistant Professor at the Department of Electrical and Electronics Engineering of Ankara University when this research was carried.

\end{document}